\documentclass[a4paper]{article}
\usepackage{amsmath,amsfonts,amsthm, amssymb}
\usepackage[all]{xy}
\newtheorem{theorem}{Theorem}[section] 
\newtheorem{lemma}[theorem]{Lemma} 
\newtheorem{proposition}[theorem]{Proposition}

\newtheorem{definition}[theorem]{Definition}

\newcommand{\newell}{\hat{\ell}}
\newcommand{\R}{{\mathbb R}}
\newcommand{\Z}{{\mathbb Z}}
\newcommand{\N}{{\mathbb N}}
\newcommand{\CC}{{\mathbb C}}
\newcommand{\C}{{\mathcal C}}
\newcommand{\F}{{\mathcal F}}
\newcommand{\Sym}{\mathrm{Sym}}

\newcommand{\diam}{\mathrm{diam}}

\newenvironment{mylist}{\begin{list}{}{
\setlength{\parskip}{0mm}
\setlength{\topsep}{2mm}
\setlength{\parsep}{0mm}
\setlength{\itemsep}{0.5mm}
\setlength{\labelwidth}{7mm}
\setlength{\labelsep}{3mm}
\setlength{\itemindent}{0mm}
\setlength{\leftmargin}{12mm}
\setlength{\listparindent}{6mm}
}}{\end{list}}
\title{Some closure results for $\C$-approximable groups}
\author{Derek F. Holt and Sarah Rees}
\date{}
\begin{document}
\maketitle
\begin{abstract}
We investigate closure results for $\C$-approximable groups,
for certain classes $\C$ of groups with invariant length functions. 
In particular we prove, each time for certain (but not necessarily the same)
classes $\C$ that: \linebreak
(i) the direct product of two $\C$-approximable groups is
$\C$-approximable; 
(ii) the restricted standard wreath product $G \wr H$ is $\C$-approximable
when $G$ is $\C$-approximable and $H$ is residually finite; and
(iii) a group $G$ with normal subgroup $N$ is $\C$-approximable
when $N$ is $\C$-approximable and $G/N$ is amenable.
Our direct product result is valid for LEF, weakly sofic and hyperlinear
groups, as well as for all groups that are approximable by finite
groups equipped with commutator-contractive invariant length functions
(considered in \cite{Thom}). Our wreath product result is valid for weakly
sofic groups, and we prove it separately for sofic groups.
We note that this last result has recently been generalised by Hayes and Sale,
who prove in \cite{HayesSale} that the restricted standard wreath product of
any two sofic groups is sofic. Our result on
extensions by amenable groups is valid for weakly sofic groups, and was proved
in \cite[Theorem 1 (3)]{ElekSzabo} for sofic groups $N$.
\end{abstract}



\bigskip

\section{Introduction}
\label{sec:intro}
Our interest in $\C$-approximable groups stems from the fact that,
by making an appropriate choice of the class $\C$,
the definition of a $\C$-approximable group equates to that of one of
a variety of classes of groups currently of interest, including
sofic groups, hyperlinear groups, weakly sofic groups,  linear sofic groups,
and LEF groups. Hence techniques that apply to one such class
can often be applied to another. In this article we develop some general
techniques to establish some closure properties for many of these classes,
specifically for direct products, for wreath products with residually
finite groups, and for extensions by amenable groups.
We shall refer to closure results in the literature, mostly for specific
classes of $\C$-approximable groups; in some cases our proofs have been inspired
by the proofs of those.
We are grateful to the anonymous referee of the paper for a careful reading
and several helpful comments and corrections.

Our definition of a $\C$-approximable group is taken from
\cite[Definition 1.6]{Thom} and specialises to the definitions of sofic
and hyperlinear groups in \cite{Capraro};
we shall discuss some of the alternative definitions later on in this section.
Our definition requires the concept of an {\em invariant length function} on a
group $K$; that is, a map
\linebreak
$\ell:K \rightarrow [0,1]$ such that, for all $x,y \in K$:
\begin{eqnarray*}
&\ell(x)=0 \iff x=1,\quad &\ell(x^{-1})= \ell(x),\\
&\ell(xy) \leq \ell(x) + \ell(y),& \ell(xyx^{-1})=\ell(y).
\end{eqnarray*}
Every group admits the trivial length function $\ell_0$ defined by
$\ell_0(x)=1$ if $x\neq 1$, $\ell_0(1)=0$, and may admit many others.
The Hamming norm, which computes the proportion of points moved by a
permutation of a finite set, gives an invariant length function for finite
symmetric groups.

In the following definition $\C$ is understood to be a set of pairs, each pair
consisting of a group $K$ together with an invariant length function $\ell_K$
on $K$; so the same group may occur in $\C$ with more than one length
function. For a group $K$, 
the statement $K \in \C$ means that $K$ is the group in at least
one such pair.

\begin{definition}\label{def:defs}
\begin{enumerate}
\item
For a group $G$, a map $\delta:G \rightarrow \R$ (for which we write
$\delta_g$ rather than $\delta(g)$) is a {\em weight function for $G$} if
$\delta_1=0$ and  $\delta_g>0$ for all $1 \ne g \in G$.
\item
Let $G$ be a group with weight function $\delta$, let 
$K$ be a group with invariant length function $\ell_K$,
let $\epsilon > 0$, and let $F$ be a finite subset of $G$.
Then the map $\phi:G \to K$ is a {\em $(F,\epsilon,\delta,\ell_K)$
quasi-homomorphism} if:
\begin{mylist}
\item $\phi(1)=1$;
\item $\forall g,h \in F$,
           $\ell_K(\phi(gh)\phi(h)^{-1}\phi(g)^{-1}) \leq \epsilon$; and
\item $\forall g \in F\setminus \{1\}$, $\ell_K(\phi(g)) \geq \delta_g.$
\end{mylist}
\item
Let $\C$ be a class of groups with associated invariant length functions.
Then a group $G$ is {\em $\C$-approximable} if it has a weight function
$\delta$, such that, for each $\epsilon > 0$ and for each finite subset $F$
of $G$, there exists an $(F,\epsilon,\delta,\ell_K)$-quasi-homomorphism
$\phi:G \to K$ for some $(K,\ell_K) \in \C$.
\end{enumerate}
\end{definition}

Since these conditions cannot possibly be satisfied if $\delta_g > 1$ for some
$g \in G$, we shall always assume that $\delta_g \le 1$.

In particular, sofic groups are precisely those groups that are
$\C$-approximable with respect to the class  $\C$ of finite symmetric groups
with length function defined by the Hamming norms, and with weight functions
of the form $\delta_g = c$
for all $1 \ne g \in G$, for some fixed constant $c>0$; see  
\cite[Theorem 5.2]{PestovKwiatkowska}.

The (normalised) Hilbert-Schmidt norm on the set of $n \times n$
complex matrices $A = (a_{ij})$ is defined by
\[||(a_{ij})||_{{\rm HS}_n}:= \sqrt{\frac{1}{n} \sum_{i,j}|a_{ij}|^2} =
   \sqrt{\frac{1}{n} {\rm Tr}(A^*A)}.\]
The hyperlinear groups are precisely those groups that are $\C$-approximable
with respect to the class  $\C$ of finite dimensional unitary groups with
length function defined by $\ell(g) = \frac{1}{2}||g-I_n||_{{\rm HS}_n}$, and
with the same weight functions as for sofic groups;
see \cite[Theorem
4.2]{PestovKwiatkowska}.  
Furthermore, weakly sofic groups, linear sofic groups and LEF groups can all
be defined as $\C$-approximable groups, where
the classes $\C$ are (respectively)
the class $\F$ of all finite groups equipped with all associated invariant
length functions, the groups ${\rm GL}_n(\CC)$
equipped with the norm $\ell(g) = \frac{1}{n}{\rm rk}(I_n-g)$
\cite{ArzhantsevaPaunescu}, and
the finite groups equipped with the trivial length function.
We refer the reader to \cite{ArzhantsevaGal,CiobanuHoltRees, ElekSzabo,
ElekSzabo2,Paunescu,Stolz}
for a number of closure results involving various of these classes of groups.

Following \cite{Thom} we say that an invariant length function
$\ell:K \rightarrow [0,1]$ is {\em commutator-contractive} if it satisfies the
condition
\[ \ell([x,y]) \leq 4 \ell(x)\ell(y),\quad\forall x,y \in K.\]
Note that the trivial length function is commutator-contractive.
Let $\F_C$ be the class of all finite groups, each equipped with all
commutator-contractive length functions.  The main result of \cite{Thom} is
that Higman's group \cite{Higman} is not $\F_C$-approximable.  This group is
widely seen as a candidate for a first example of a non-sofic group.

There are many variations in the literature of the definition of
a $\C$-approximable group, not all of which are believed to be equivalent
in general to our basic definition, although the paucity of known examples of
groups that are not $\C$-approximable makes it difficult to prove their
inequivalence. 

Some definitions, such as \cite[Definitions 1,2]{Glebsky} and
\cite[Section 2]{Stolz} allow invariant length functions to take values in
$[0,\infty)$ rather than in $[0,1]$. This does not affect the classes of
sofic, hyperlinear, linear sofic and LEF groups, since the length functions
used in these classes all have range $[0,1]$.  It is also easily seen that the
class of weakly sofic groups is not changed by this variant since, if a group
is weakly sofic using length functions with range $[0,\infty)$, and $\ell_K$
is such a length function on a finite group $K$, then simply by replacing
$\ell_K(g)$ by the new length function $\max(\ell_K(g),1)$, we can show that
$G$ is weakly sofic using length functions with range $[0,1]$.
So this variation in the range of permissible length functions
does not appear to us to be significant.

The more substantial variants involve the condition
\[ \forall g \in F,\, \ell_K(\phi(g)) \geq \delta_g\]
in the definition of $\C$-approximability. These are discussed in
\cite[Section 2]{Stolz}.
The group $G$ is said to have the {\em discrete $\C$-approximation property}
if the weight function for $G$ can be chosen to be constant on all
non-identity elements. It is said to have the {\em strong discrete
$\C$-approximation property} if the condition above is replaced by
\[ \forall g \in F,\, \ell_K(\phi(g)) \geq \diam(K) - \epsilon \]
where $\diam(K)$ is defined to be $\sup \{ \ell_K(x) : x \in K \}$,
and $\epsilon$ is as in Definition~\ref{def:defs}(3).
By choosing the weight function $\delta_g = \diam(G)/2$ for all $g \in G
\setminus \{1\}$, we see immediately that the strong discrete
$\C$-approximation property implies the discrete
$\C$-approximation property, which clearly implies that $G$ is
$\C$-approximable using our definition. But the converse implications are not
clear, and may not hold in general.

The definition given for sofic groups in \cite{ElekSzabo} enforces the strong
discrete approximation property. But it is shown in \cite[Exercise
II.1.8]{Capraro} that, for this class, any $\C$-approximable group
has the strong discrete $\C$-approximation property.

It is proved in \cite[Proposition 5.13]{ArzhantsevaPaunescu} that
linearly sofic groups have the discrete $\C$-approximation property, but
it appears to be unknown whether they have the
strong discrete $\C$-approximation property.

Hyperlinear groups do not have the strong $\C$-approximation property,
and we are grateful to the referee for pointing this out to us.
The diameter of the unitary group $\mathcal{U}(n)$ with length function
defined as above by $\ell(g) = \frac{1}{2}||g-I_n||_{{\rm HS}_n}$ is $1$.
By using the identity
$$||g-h||^2_{{\rm HS}_n} + ||g+h||^2_{{\rm HS}_n} = 4$$ 
for $g,h \in \mathcal{U}(n)$ and putting $h=I_n$, we see that,
if $1 - \ell(g)$ is small, then $g$ is close to $-I_n$ with respect to the
Hilbert-Schmidt metric. So if $1-\ell(g_1)$ and $1-\ell(g_2)$ are both small,
then $g_1g_2$ is close to $I_n$ and hence $\ell(g_1g_2)$ is close to $0$.
It follows that a hyperlinear group with the strong discrete
$\C$-approximation property must be finite with order at most $2$.

What is true for hyperlinear groups is that, for any finite $F
\subseteq⊆ G$ and ε$\epsilon > 0$, there exists an approximately
multiplicative map $\phi: G \to \mathcal{U}(n)$ for which
$|{\rm Tr}(\phi(g))/n| < \epsilon$ for all $g \in F \setminus \{1\}$.
This was first proved by Elek and Szabo in \cite{ElekSzabo3} using ideas
introduced by R\u{a}dulescu in \cite{Radulescu}.

It is not difficult to show that the classes of $\F$-approximable (i.e. weakly
sofic) and $\F_C$-approximable groups both have the strong discrete
$\C$-approximation property. For a finite subset $F$ of a group $G$
in one of these two classes, and $\epsilon > 0$, let $c =
\min \{ \delta_g : g \in F \}$, and let $\phi:G \to K$ be a 
$(F,c\epsilon,\delta,\ell_K)$-quasi-homomorphism. Then, by
replacing $\ell_K$ by the length function
$\ell'_K(x) := \min(\ell_K(x)/c,1)$, which is commutator-contractive if
$\ell_K$ is, we see that $\phi$ is a $(F,\epsilon,\delta,\ell'_K)$
quasi-homomorphism for which $\ell'_K(\phi(g)) = 1$ for all $g \in F$,
so $G$ has the strong discrete $\C$-approximation property.

We prove our closure results for direct products, wreath products,
and extensions by amenable groups
in Sections \ref{sec:dp}, \ref{sec:wp} and \ref{sec:wps}, and \ref{sec:extam},
respectively. To prove the last of these, on extensions of
$\C$-approximable groups $N$ by amenable groups, we need to assume that
the group $N$ has the discrete $\C$-approximation property.
For each of our closure results, it is straightforward to show
that, if the groups that are assumed to be $\C$-approximable have the discrete
or the strong discrete $\C$-approximation property, then so does the
group $G$ that is proved to be $\C$-approximable.

Concerning free products, we note that it is proved in
\cite[Theorem 1]{ElekSzabo}, \cite[Theorem 5.6]{Stolz} and
\cite{Popa, Voiculescu}, respectively,
that the classes of sofic, linear sofic and hyperlinear groups are closed under free
products; further it is proved in \cite{BrownDykemaJung} that free products of hyperlinear groups amalgamated over amenable subgroups are hyperlinear.
We thank the referee for bringing to our attention the results 
for hyperlinear groups.
We are unaware of any corresponding results
for weakly sofic groups, and our efforts to prove such a
result have so far been unsuccessful.

\section{The direct product result}
\label{sec:dp}
In order to state and prove our closure result for direct products
of $\C$-approximable groups, we need to construct 
an appropriate invariant length function for the direct product of two groups
in $\C$.  Suppose that $(J,\ell_J),(K,\ell_K) \in \C$.
Then, for $p \in \N \cup \{\infty\}$,
we define the functions $L^p_{\ell_J,\ell_K}: J \times K \rightarrow [0,1]$ by
\[ L^p_{\ell_J,\ell_K}(x,y) := \sqrt[p]{\frac{\ell_J(x)^p +
\ell_K(y)^p}{2}},\quad p \in \N \]
and $L^\infty_{\ell_J,\ell_K}(x,y) := \max(\ell_J(x),\ell_K(y))$.
We write just $L^p(x,y)$ when there is no ambiguity.

Note that $L^p(x,y) \le L^\infty(x,y) \leq 1$ for all $p \ge 1$.

It follows immediately from Minkowski's inequality (basically the triangle
inequality for the $L^p$ norm) that $L^p$ satisfies the rule
\[ L^p(x_1x_2,y_1y_2) \leq L^p(x_1,y_1) + L^p(x_2,y_2), \]
and hence is an invariant length function on $J\times K$.
As we shall see below, we can use $L^p$ (for some choice of $p$) to deduce the
closure of $\C$-approximable groups under direct products provided that
$(J\times K,L^p) \in \C$.

\begin{theorem}
\label{thm:directprod}
Let $\C$ be a class of groups with associated invariant length functions
and suppose that, for some fixed $p \in \N \cup \{\infty\}$, and for any groups
$J,K \in \C$,
\[(J,\ell_J),(K,\ell_K) \in \C \Rightarrow (J \times K, L^p) \in \C.\]
Then the direct product $G \times H$ of two $\C$-approximable groups $G$ and
$H$ is $\C$-approximable.
\end{theorem}
\begin{proof}
Suppose that $\C,p$ satisfy the conditions of the theorem.

Let $G$ and $H$ be $\C$-approximable with associated weight functions
$\delta^G$ and $\delta^H$.  We define the weight function
$\delta^{G \times H}$ by
\[ \delta^{G \times H}((g,h)) :=
         \sqrt[p]{\frac{\delta^G(g)^p + \delta^H(h)^p}{2}}. \]

Now suppose that $\epsilon>0$ is given, and let  $F$ be a finite subset
of $G \times H$.
Then we can find finite subsets $F_G \subseteq G$, $F_H \subseteq H$ such
that $F \subseteq F_G \times F_H$, pairs $(J,\ell_J),(K,\ell_K) \in \C$,
an $(F_G,\epsilon,\delta^G,\ell_J)$-quasi-homomorphism
$\phi_G: G \rightarrow J$, and an
$(F_H,\epsilon,\delta^H,\ell_K)$-quasi-homomorphism
$\phi_H:H \rightarrow K$.

We define $\phi: G \times H \rightarrow M:= J \times K$ by
$\phi((g,h)) :=(\phi_G(g),\phi_H(h))$ and
$\ell_M((x,y)) := L^p((x,y))$.

We verify easily that, for $(g_1,h_1),(g_2,h_2) \in F$, and hence
$g_1,g_2 \in F_G$, $g_2,h_2 \in F_H$,
\begin{eqnarray*}
&&\ell_M(\phi((g_1g_2,h_1h_2)\phi(g_2,h_2)^{-1}\phi(g_1,h_1)^{-1})\\
&&= L^p((\phi_G(g_1g_2)\phi_G(g_2)^{-1}\phi_G(g_1)^{-1},
\phi_H(h_1h_2)\phi_H(h_2)^{-1}\phi_H(h_1)^{-1})) \le \epsilon,
\end{eqnarray*}
and the other conditions are similarly verified.
\end{proof}

We can apply the result to deduce closure under direct products
for the classes of weakly sofic groups, LEF groups, hyperlinear groups, linear
sofic groups and Thom's class of $\F_C$-approximable groups \cite{Thom}.

For weakly sofic groups, the condition holds for any $p$, and for LEF groups
it holds for $p=\infty$. 

When $\ell_J, \ell_K$ are Hilbert-Schmidt norms in the same
dimension $n$, the function $L^2$ matches the Hilbert-Schmidt norm in
dimension $2n$; observing that whenever $G$ maps by a quasi-homomorphism to a
linear group in dimension $m$ it also maps to a linear group in dimension $rm$,
for any $r$, via a quasi-homomorphism with the same parameters (the composite
of the original quasi-homomorphism and a diagonal map), we see that in essence
the theorem applies with $p=2$ to prove closure under direct products
for the class of hyperlinear groups. Similarly it applies when $p=1$
to prove closure under direct products for the class of linear sofic groups.

But for Hamming norms $\ell_J,\ell_K$, the function $L^p_{\ell_J,\ell_K}$ is not a Hamming
norm, and hence we cannot deduce the closure of the class of sofic groups
under direct products from this result.

Of course all of these specific closure results are already known, and the
corresponding result for sofic groups is proved in \cite{ElekSzabo}.

The following lemma together with Theorem \ref{thm:directprod} shows that the
class of $\F_C$-approximable groups is closed under direct products.

\begin{lemma}
Suppose that groups $J,K$ have commutator-contractive length functions
$\ell_J:J \rightarrow [0,1]$, $\ell_K:K \rightarrow [0,1]$.
Then $L^\infty$, as defined above,
is a commutator-contractive length function for their direct product. 
\end{lemma}
\begin{proof}
Let $(g_1,h_1),(g_2,h_2) \in G \times H$. Then
\begin{eqnarray*}
&&L^\infty([(g_1,h_1),(g_2,h_2)]) = L^\infty(([g_1,g_2],[h_1,h_2])) \\
&&= \max(l_J([g_1,g_2]),l_K([h_1,h_2])) \le
\max(4l_J(g_1)l_J(g_2),4l_K(h_1)l_K(h_2))\\
&&\le 
4\max(l_J(g_1),l_K(h_1)) \max(l_J(g_2),l_K(h_2)) =
4L^\infty((g_1,h_1))L^\infty((g_2,h_2)).
\end{eqnarray*}
\end{proof}

This result does not hold in general for $L^p$ with $p \in [1,\infty)$.

\section{The wreath product result}
\label{sec:wp}
By definition the restricted standard wreath product $W= G \wr H$ of two
groups $G,H$ is a semi-direct product $H \ltimes B$. The {\em base group} $B$
of $W$ is the direct product of copies of $G$, one for each $h \in H$,
and is viewed as the set of all functions $b:H \rightarrow G$ with finite
support (that is, with $b(h)$ trivial for all but finitely many $h\in H$)
Elements of $B$ are multiplied component-wise; that is, $b_1b_2(h)=b_1(h)b_2(h)$
for $b_1,b_2 \in B$, $h \in H$.
For $b \in B$, we denote by $b^{-1}$ the function in $B$ defined by
$b^{-1}(h) = b(h)^{-1}$. 
The (right) action of $H$ on $B$ is defined by the rule
$b^h(h')=b(h'h^{-1})$; we often abbreviate $(b^h)^{-1}=(b^{-1})^h$ as $b^{-h}$.
So the elements of $W$ have the form $hb$ with $h \in H$, $b \in B$, and
$(h_1b_1)(h_2b_2) = h_1h_2b_1^{h_2}b_2$, while $(h,b)^{-1} = (h^{-1},b^{-h^{-1}})$. 

In order to state and prove our closure result for wreath products of
$\C$-approximable groups, we need to construct an appropriate invariant length
function for the wreath product $J \wr X$ of a group $J \in \C$ by a finite
group $X$.

Where $B'$ is the base group of $J \wr X$,
we define $\ell_J^X:J \wr X \rightarrow [0,1]$ as follows. For $b' \in B'$,
we put
\[ 
\ell_L^X(b') = \max_{x \in X} \ell_J(b'(x)), \]
and then, for $x \neq 1$, put
\[ \ell_J^X(xb')= 1.\]
It is straightforward to verify that $\ell_J^X$ is an invariant length function.

\begin{theorem}
\label{thm:wreathprod}
Let $\C$ be a class of groups with associated invariant length functions
and suppose that, for all $(J,\ell_J) \in \C$ and all finite groups $X$,
the wreath product $(J \wr X,\ell_J^X)$ is in $\C$.
Suppose that the group $G$ is $\C$-approximable and the group $H$ is
residually finite. Then the restricted standard wreath product $G \wr H$ is
$\C$-approximable.
\end{theorem}

\begin{proof}
Suppose that $G$ is $\C$-approximable with associated weight function
$\delta$, and that $H$ is residually finite, and let
$W = G \wr H$ be the restricted standard wreath product.
Let $B$ be the base group.

We define the weight function $\beta: W \rightarrow \R$  as follows:
\[ \beta_{hb}= \left \{ \begin{array}{ll} 1 & \hbox{if}\quad h\neq 1 \\
                                \max_{k \in H} \delta_{b(k)} & \mbox{otherwise.}
\end{array} \right .\]

Let $\epsilon>0$ be given, and let $F = \{ h_ib_i : 1 \le i \le r \}$ be a
finite subset of $W$.  Our aim is to find $(K,\ell_K) \in \C $ and an
$(F,\epsilon,\beta_W,\ell_K)$-quasi-homomorphism $\psi:W \rightarrow K$.

Let $E$ be a finite subset of $H$ that contains
\begin{mylist}
\item[(i)] $h_i$ for $1 \le i \le r$;
\item[(ii)] all $h \in H$ with $b_j(h) \ne 1$ for some $j$ with
$1 \le j \le r$; and
\item[(iii)] all $h \in H$ with $b_j(hh_i^{-1}) \ne 1$ for some
$i,j$ with $1 \le i \le r$, $1 \le j \le r$.
\end{mylist}

Choose $N \unlhd H$ with $H/N$ finite such that the images in $H/N$ of the
elements of $E$ are all distinct and the images of $E \setminus \{1\}$
are nontrivial.

Let $D = \{ b_j(h) : 1 \le j \le r, h \in H \}$.  Then $D$ is a finite subset
of $G$ so, by our definition of $\C$-approximability,
for a  given $\epsilon > 0$, there exists $(J,\ell_J)\in \C$,
and a $(D,\epsilon,\delta,\ell_J)$-quasi-homomorphism $\phi:G\rightarrow J$.

We will approximate $W$ by $K :=J \wr (H/N)$, and 
let $\ell_K$ be the length function $\ell_J^{H/N}$ defined above.
Let $B'$ be the base group of $K$,
that is, the group of finitely supported functions from $H/N$ to $J$.

We define $\psi: W \to K$ as follows. Suppose that $b \in B$, and $h,k \in H$.
Note that our choice of $N$ ensures that $E \cap kN$ is either empty or
consists of a single element $k' \in kN$.  We let
$\psi(hb) := \bar{h}\hat{b}$, where we write $\bar{h}$ for $hN$ and
$\hat{b}: H/N \rightarrow J$ is defined by the rule 
\[ \hat{b}(kN)= \left \{
\begin{array}{lll} 1 &\hbox{\rm when} & E \cap kN=\emptyset\\
\phi(b(k')) & \hbox{\rm when}  & E \cap kN=\{ k'\}. \\ \end{array} \right .
\]
We claim that $\psi$ has the appropriate properties.  Certainly $\psi(1)=1$.

We first verify the required lower bound on $\ell_K(\psi(hb))$ for elements
$hb \in F$. If $h \neq 1$ then our choice of $N$ ensures that $\bar{h} \neq 1$,
and so $\ell_K(\psi(hb))=1=\beta_{hb}$.

If $h=1$, then (where the maximum of an empty set of numbers in $[0,1]$
is defined to be $0$),
\begin{eqnarray*}
\ell_K(\psi(hb)) &=& \ell_K(\psi(b))=\ell_K(\hat{b})\\
&=&\max_{kN \in H/N: \{k'\}:= kN \cap E \neq \emptyset} \ell_J(\phi(b(k')))\\
&=& \max_{k' \in E} \ell_J(\phi(b(k')))\\
&\geq&  \max_{k' \in E} \delta_{b(k')}
= \max_{k' \in H} \delta_{b(k')}= \beta_b. \end{eqnarray*}

The equality of the two maxima in the final line follows from the definition
of $E$, which ensures that $b(k)=1$ for any $k \in H\setminus E$ and hence
that, for such $k$, $\delta_{b(k)}=0$.

It remains to show that, for $h_ib_i,h_jb_j \in F$,
\[ l_K(\psi(h_ib_ih_jb_j)(\psi(h_ib_i)\psi(h_jb_j))^{-1}) \le \epsilon.\]
We have
\begin{eqnarray*}
\psi(h_ib_ih_jb_j)&=&\psi(h_ih_jb_i^{h_j}b_j)=
   \overline{h_ih_j}\widehat{b_i^{h_j}b_j}, \quad \mbox{and}\\
\psi(h_ib_i)\psi(h_jb_j) &=& (\bar{h}_i\hat{b}_i)(\bar{h}_j\hat{b}_j)=
   \bar{h}_i\bar{h}_j\hat{b}_i^{\bar{h}_j}\hat{b}_j.
\end{eqnarray*}
Since $l_K$ is invariant under conjugation, the length we need is that of the
element
\[ b' :=\widehat{b_i^{h_j}b_j}\hat{b}_j^{-1}(\hat{b}_i^{\bar{h}_j})^{-1} \] 
of $B'$.
By definition, $\ell_K(b') = \max_{kN \in H/N} \ell_J(b'(kN))$.
So choose a coset $kN$. We want to bound $\ell_J(b'(kN))$ for each such choice.
We have
\begin{eqnarray*}
b'(kN) &=& \widehat{b_i^{h_j}b_j}(kN)
        (\hat{b}_j(kN))^{-1}(\hat{b}_i^{\bar{h}_j}(kN))^{-1} \\ 
&=& \widehat{b_i^{h_j}b_j}(kN)
     (\hat{b}_j(kN))^{-1}(\hat{b}_i(kh_j^{-1}N))^{-1} \\ 
&=& \left \{ \begin{array}{ll}
     (\hat{b}_i(kh_j^{-1}N))^{-1}
                   \,&\hbox{\rm if (1):}\quad kN \cap E = \emptyset,\\
 \phi(b_i(k'h_j^{-1})b_j(k'))(\phi(b_j(k')))^{-1}\\
\quad\quad \times (\hat{b}_i(kh_j^{-1}N))^{-1} 
                   &\hbox{\rm if (2):}\quad kN \cap E = \{ k'\}, \\
\end{array} \right .
\end{eqnarray*}
since in Case 1 we have
$\widehat{b_i^{h_j}b_j}(kN)= \hat{b}_j(kN)=1$,
and in Case 2, we have
$\widehat{b_i^{h_j}b_j}(kN)=
   \phi((b_i^{h_j}b_j)(k'))=\phi(b_i(k'h_j^{-1})b_j(k'))$,  
and $\hat{b}_j(kN) = \phi(b_j(k'))$. 

When $E \cap kh_j^{-1}N = \emptyset$, we have $\hat{b}_i(kh_j^{-1}N)=1$.
In that case, by the definition of $E$, we also have $b_i(k'h_j^{-1})=1$
and so, in both Case 1 and Case 2, we deduce that
$b'(kN)=1$ and $\ell_J(b'(kN))=0$. 

Otherwise $E \cap kh_j^{-1}N$ is non-empty, and its single element
is equal to $k''h_j^{-1}$, for some $k'' \in kN$.

Suppose first that $b_i(k''h_j^{-1})=1$, and hence again
we have $\hat{b}_i(kh_j^{-1}N)=1$.
If we are in Case 2 then we must also have $b_i(k'h_j^{-1})=1$,
since if $b_i(k'h_j^{-1})\neq 1$,
then  Condition (ii) of the
definition of $E$ gives $k'h_j^{-1} \in E$, and so $k'=k''$, 
contradicting $b_i(k''h_j^{-1})=1$. Then, just as above, we see that
in both Cases 1 and 2 we again get $b'(kN)=1$ and $\ell_J(b'(kN))=0$.

Otherwise $b_i(k''h_j^{-1})\ne 1$ and Condition (iii) of the definition of
$E$ gives $k'' \in E$ and hence we are in Case 2 with $k'=k''$. Then
\[ b'(kN) = 
\phi(b_i(k'h_j^{-1})b_j(k'))  
\phi(b_j(k')^{-1}\phi(b_i(k'h_j^{-1}))^{-1}.\]
Since $\phi$ was assumed to be a
$(D,\epsilon,\delta,\ell_J)$-quasi-homomorphism,
we have \linebreak $\ell_J(b'(kN)) \le \epsilon$ and, since this is true for all
$kN \in H/N$, we get $\ell_K(b') \le \epsilon$ as required.
\end{proof}

The conditions of the theorem clearly hold for the class $\F$, as well as for
finite groups equipped with the trivial length function, and hence the classes
of weakly sofic and LEF groups are both closed under restricted wreath products
with residually finite groups.
The following lemma together with Theorem \ref{thm:directprod}
shows that the class of $\F_C$-approximable groups is also closed under 
restricted wreath products with residually finite groups.
\begin{lemma}
\label{lem:wpcc}
Let $J$ be a group equipped with an invariant function $\ell_J$.
If $\ell_J$ is commutator-contractive, then so is $\ell_J^X$,
for any finite group $X$.
\end{lemma}
\begin{proof}
We consider the commutator of two elements $x_1b_1$ and $x_2b_2$ in $J$.

First suppose that $x_1$ and $x_2$ are both non-trivial.  In this case
$\ell_J^X(x_1b_1)=\ell_J^X(x_2b_2)=1$ and so the inequality holds trivially.

Now suppose that $x_1=x_2=1$.
Then 
\begin{eqnarray*}
\ell_J^X([b_1,b_2]) &=& \max_{x \in X}\ell_J([b_1,b_2](x))
= \max_{x \in X} \ell_J([b_1(x),b_2(x)]) \\
&\leq& 4\max_{x \in X} \ell_J(b_1(x))\ell_J(b_2(x))\\
&\leq& 4 \max_{x \in X} \ell_J(b_1(x)) \max_{y \in X} \ell_J(b_2(y))
= 4 \ell_J^X(b_1) \ell^X_J(b_2)
\end{eqnarray*}

Finally suppose that $x_1=1$, $x_2 \neq 1$ (the other case is very
similar).  Then
\begin{eqnarray*}
\ell_J^X([b_1,x_2b_2]) &=& \ell_J^X(b_1^{-1}b_2^{-1}x_2^{-1}b_1x_2b_2)
=\ell_J^X(b_1^{-1}b_2^{-1}b_1^{x_2}b_2)\\
&=& \max_{x \in X} \ell_J(b_1(x)^{-1}b_2(x)^{-1}b_1^{x_2}(x)b_2(x))\\
&=& \max_{x \in X} \ell_J(b_1(x)^{-1}b_2(x)^{-1}b_1(xx_2^{-1})b_2(x))\\
&\le&  \max_{x \in X} (\ell_J(b_1(x)^{-1})+\ell_J(b_2(x)^{-1}b_1(xx_2^{-1})b_2(x)))\\
&=&  \max_{x \in X} (\ell_J(b_1(x)^{-1})+\ell_J(b_1(xx_2^{-1})))\\
&\leq&  \max_{x \in X} (\ell_J(b_1(x)^{-1}))+\max_{y \in X}(\ell_J(b_1(y)))\\
&\leq&  2\max_{x \in X} (\ell_J(b_1(x)^{-1}) = 2 \ell_J^X(b_1)\\
\end{eqnarray*}
\end{proof}

\section{The wreath product result for sofic groups}
\label{sec:wps}
We prove now the corresponding result for sofic groups. For this, we are not
free to choose our own norm function on the wreath product, but we must use the
Hamming distance norm. The proof is nevertheless very similar in structure
to that of Theorem \ref{thm:wreathprod}. 
We use the definition of sofic groups given in \cite{ElekSzabo} where,
rather than having a weight function on the group $G$, we require that,
for finite $F \subseteq G$, the proportion of moved points of elements of
$F \setminus \{ 1 \}$ in a $(F,\epsilon)$-quasi-action of $G$ on a finite set
is at least $1 - \epsilon.$

We note that this result has recently been generalised by Hayes and Sale,
who prove in \cite{HayesSale} that the restricted standard wreath product of
any two sofic groups is sofic.

\begin{theorem}
\label{thm:wreathprodsofic}
The restricted standard wreath product $G \wr H$ of a sofic group $G$ and
a residually finite group $H$ is sofic.
\end{theorem}

\begin{proof}
Assume that $G$ is sofic and $H$ is residually finite, and let
$W = G \wr H$ be the restricted standard wreath product. So,
as in the proof of Theorem \ref{thm:wreathprod}, $W$ is the semidirect product
of its base group $B$ by $H$.

Let $F = \{ h_ib_i : 1 \le i \le r \}$ be a finite subset of $W$.
Then, for a given $\epsilon>0$, we need to find a
$(F,\epsilon)$-quasi-action of $W$ on some finite set $Y$.

We define the finite subset $E$ of $H$, the normal subgroup $N$ of $H$,
and the finite subset $D$ of $G$ exactly as in the proof of
Theorem \ref{thm:wreathprod}. So, in particular, for any $k \in H$,
$E \cap kN$ is either empty or consists of a single element $k' \in kN$.
Let $m = |H/N|$.

Then, by \cite[Lemma 2.1]{ElekSzabo},
for a  given $\epsilon > 0$, there is a $(D,\epsilon/m)$-quasi-action
$\phi:G \rightarrow \Sym(X)$ of $G$ on some finite set $X$,
and we may assume that $\phi(1)=1$. Since we can choose both $m$ and
$X$ to be arbitrarily large for given $D$ and $\epsilon$,
we may assume that $|X|^{-m/2} < \epsilon$.

Let $Y = X^{H/N}$ be the set of functions $\delta:H/N \to X$.
So $|Y| = |X|^m$.  We define $\psi: W \to \Sym(Y)$ as follows.
(The image of $\psi$ is contained in the primitive wreath product of
$\Sym(X)$ and $H/N$, as defined in \cite[Section 2.6]{DixonMortimer}.)

For $b \in B$, $h,k \in H$, 
$\delta^{\psi(hb)}(kN) := \delta(kh^{-1}N)^{\tau(b,k)}$,  where
\[
  \tau(b,k) := \left \{
\begin{array}{lll} 1 &\hbox{\rm when} & E \cap kN=\emptyset\\
\phi(b(k')) & \hbox{\rm when}  & E \cap kN=\{ k'\}. \\ \end{array} \right .
 \]

We claim that $\psi$ is a $(F,\epsilon)$-quasi-action of $W$ on $Y$.
Observe first that $\psi(1)=1$.

We check next that, for each $h_ib_i \in F \setminus \{ 1 \}$,
$\psi(h_ib_i)$ is $(1-\epsilon)$-different from $1$.
If $h_i \ne 1$ then, by assumption, $h_i \not\in N$, so $kh_i^{-1}N \ne
kN$ for all $kN \in H/N$. So, if $\delta \in Y$ is a fixed point of
$\psi(h_ib_i)$, then  the value of $\delta(kN)$ is uniquely determined by that
of $\delta(kh_i^{-1}N)$ for each $kN \in H/N$, so the proportion
of fixed points is at most $|X|^{m/2}/|X|^m = |X|^{-m/2}$, which we
assumed to be less than $\epsilon$.

If, on the other hand, $h_i=1$ and $b_i \ne 1$, then there exists
$h \in E$ with $b_i(h) \ne 1$. Now an element $\delta \in Y$ is fixed by
$\psi(h_ib_i) = \psi(b_i)$ if and only if $\delta(kN)$ is fixed by $\tau(b,k)$
for all $kN \in H/N$. Hence, in particular, for a fixed point $\delta$,
we have $\delta(hN) = \delta(hN)^{\tau(b_i,h)}$, and so $\delta(hN)$ is
a fixed point of $\tau(b_i,h) = \phi(b_i(h))$.
Since the proportion of such points in $X$ is, by assumption, at
most $\epsilon$, the same is true for $\psi(b_i)$. 

Finally we need to verify that $\psi(h_ib_i)\psi(h_jb_j)$ is
$\epsilon$-similar to $\psi(h_ih_jb_i^{h_j}b_j)$ for each $i,j$
with $1 \le i,j \le r$; that is, that the two permutations
agree on at least a proportion $1-\epsilon$ of the points.

Now \[ \delta^{\psi(h_ib_i)\psi(h_jb_j)}(kN) =
(\delta^{\psi(h_ib_i)}(kh_j^{-1}N))^{\tau(b_j,k)} =
\delta(kh_j^{-1}h_i^{-1}N)^{\tau(b_i,kh_j^{-1})\tau(b_j,k)},\]
and
\[ \delta^{\psi(h_ih_jb_i^{h_j}b_j)}(kN) =
  \delta(kh_j^{-1}h_i^{-1}N)^{\tau(b_i^{h_j}b_j,k)},\]
so we need to compare $\tau(b_i,kh_j^{-1})\tau(b_j,k)$ with
$\tau(b_i^{h_j}b_j,k)$.

The argument is very similar to that in the analogous part of the proof of Theorem~\ref{thm:wreathprod}
We are in one of two cases. Either
\begin{mylist}
\item[(1)] $E \cap kN = \emptyset$, 
in which case $\tau(b_j,k) = \tau(b_i^{h_j}b_j,k)=1$, or
\item[(2)] $E \cap kN = \{k'\}$,
for some $k' \in K$, 
and so
$\tau(b_j,k)=\phi(b_j(k'))$, and $\tau(b_i^{h_j}b_j,k) =
\phi((b_i^{h_j}b_j)(k')) = \phi(b_i(k'h_j^{-1})b_j(k'))$.
\end{mylist}

When $E \cap kh_j^{-1}N = \emptyset$, then $b_i(k'h_j^{-1})=1$ and, in both
Case 1 and Case 2, 
$\tau(b_i,kh_j^{-1})\tau(b_j,k) = \tau(b_i^{h_j}b_j,k)$.

Otherwise, $E \cap kh_j^{-1}N = \{k''h_j^{-1}\}$ for some $k'' \in kN$.

Suppose first that $b_i(k''h_j^{-1})=1$. 
If we are in Case 2 then $b_i(k'h_j^{-1})=1$,
since otherwise, just as in the proof of Theorem~\ref{thm:wreathprod},
Condition (ii) of the definition of $E$
gives  $k'h_j^{-1} \in E$, and so
$k'=k''$, and we have a contradiction. 
Hence, in both Case 1 and Case 2 we again have
$\tau(b_i,kh_j^{-1})\tau(b_j,k) = \tau(b_i^{h_j}b_j,k)$.

Otherwise $b_i(k''h_j^{-1})\ne 1$, and then, again just as
in the proof of Theorem~\ref{thm:wreathprod}, 
Condition (iii) of the definition
of $E$ gives $k'' \in E$. Hence we are in Case 2 and $k'=k''$.
Then
\[\tau(b_i,gh_j^{-1})\tau(b_j,g) = \phi(b_i(k'h_j^{-1}))\phi(b_j(k'))\]
and
\[ \tau(b_i^{h_j}b_j,g) = \phi(b_i(k'h_j^{-1})b_j(k')).\] 
Since $b_i(k'h_j^{-1}), b_j(k') \in D$, our assumption that 
$\phi$ is a $(D,\epsilon/m)$-quasi-action implies that the
proportion of the points of $X$ on which the permutations
$\phi(b_i(k'h_j^{-1})b_j(k'))$ and $\phi(b_i(k'h_j^{-1}))\phi(b_j(k'))$
have the same image is at least $1-\epsilon/m$.

It follows that the proportion of elements $\delta \in Y$ with
$\delta^{\psi(h_ib_i)\psi(h_jb_j)}(kN) = \delta^{\psi(h_ih_jb_i^{h_j}b_j)}(kN)$
is at least $1-\epsilon/m$. But $\delta^{\psi(h_ib_i)\psi(h_jb_j)} =
\delta^{\psi(h_ih_jb_i^{h_j}b_j)}$ if and only if they take the same values
on all $kN \in H/N$, and the proportion of $\delta \in Y$ for which this is
true is at least $1-\epsilon$.
\end{proof}

\section{Extensions by amenable groups}
\label{sec:extam}
In Section \ref{sec:wp} we defined the restricted standard wreath product
$G \wr H$
 of groups
$G,H$.  In this section, we shall need wreath products by permutation groups.
For a group $K$ and a finite set $A$,
we define the permutation wreath product $W = K \wr \Sym(A)$ as
$W = \Sym(A) \ltimes B$ where the base group is now the set of all functions
$b:A \to K$.
As before we define $b_1b_2(a):= b_1(a)b_2(a)$ for $b_1,b_2 \in B, a \in A$,
and we define the action of $\Sym(A)$ on $B$ by the rule 
$b^\alpha(a) = b(a^{\alpha^{-1}})$, for $\alpha \in \Sym(A), a \in A$.
Much as before, elements of the wreath product are represented as pairs 
$(\alpha,b)$ with $\alpha \in \Sym(A), b \in B$, multiplied according to the rule
$(\alpha_1,b_1)(\alpha_2,b_2) = (\alpha_1\alpha_2,b_1^{\alpha_2}b_2)$,
and with $(\alpha,b)^{-1} = (\alpha^{-1},b^{-\alpha^{-1}})$.

In general the length function for finite wreath products that we used in the
proof of Theorem~\ref{thm:wreathprod} is not suitable for the proof of Theorem~\ref{thm:amenable} below. So we need to define a different one. 

Given an invariant length function $\ell_K$ on $K$, we can define an invariant
length function $\newell_K^A$ on $W$ by
\[  \newell_K^A((\alpha,b)) = \frac{1}{|A|} \left (
\sum_{a \in A: a^\alpha = a} \ell_K(b(a)) +
\sum_{a \in A: a^\alpha \neq a} 1 \right ) \]
Most of the conditions for $\newell_K^A$ to be an invariant length function are
straightforward consequences of the conditions on $\ell_K$.
The verification of
\[ \newell_K^A((\alpha_1 \alpha_2,b_1^{\alpha_2}b_2)) \leq 
\newell_K^A((\alpha_1,b_1)) + \newell_K^A((\alpha_2,b_2)) \] 
may require a little more thought. For this,
we consider the terms corresponding to the various $a \in A$ in the three sums
that make up
$\newell_K^A((\alpha_1 \alpha_2,b_1^{\alpha_2}b_2))$,
$\newell_K^A((\alpha_1,b_1))$, and $\newell_K^A((\alpha_2,b_2))$.
We see that,
for each $a \in A$ with $a^{\alpha_1} \ne a$ or $a^{\alpha_2} \ne a$,
the term in 
$\newell_K^A((\alpha_1 \alpha_2,b_1^{\alpha_2}b_2))$ is at most $1/|A|$,
but at least one of the two non-negative terms in
$\newell_K^A((\alpha_1,b_1))$ and $\newell_K^A((\alpha_2,b_2))$ is equal to
$1/|A|$.  On the other hand, for $a \in A$ with $a^{\alpha_1}=a$ and 
$a^{\alpha_2} = a$, the 
term corresponding to $a$ in 
$\newell_K^A((\alpha_1 \alpha_2,b_1^{\alpha_2}b_2))$ is
\[ \frac{1}{|A|}\ell_K(b_1^{\alpha_2}(a)b_2(a))=
       \frac{1}{|A|}\ell_K(b_1(a)b_2(a)) \leq 
\frac{1}{|A|}(\ell_K(b_1(a)) + \ell_K(b_2(a)), \]
which is the corresponding term in
$\newell_K^A((\alpha_1,b_1)) + \newell_K^A((\alpha_2,b_2))$.

\begin{theorem}
\label{thm:amenable}
Let $\C$ be a class of groups with associated invariant length functions
and suppose that, for all $(K,\ell_K) \in \C$ and all finite sets $A$,
the wreath product $(K \wr \Sym(A),\newell_K^A)$ is in $\C$.
Suppose that the group $G$ has a normal subgroup $N$  with the discrete
$\C$-approximation property (as defined in Section~\ref{sec:intro}) such that
$G/N$ is amenable. Then $G$ has the discrete $\C$-approximation property.
\end{theorem}
This result is already proved for sofic groups \cite[Theorem 1 (3)]{ElekSzabo}
and linear sofic groups \cite[Theorem 5.3]{Stolz}.  However, in order to avoid
confusion we should comment that, while the above result considers extensions
$G$ of $\C$-approximable normal subgroups $N$ with $G/N$  amenable, by
contrast, \cite[Theorem 7]{ArzhantsevaGal} considers
extensions $G$ of finitely generated residually finite normal subgroups $N$
for which $G/N$ is in a selected class ${\cal R}$ of groups (including groups
that are residually amenable groups, LEF, LEA, sofic or surjunctive) .

\begin{proof}
The proof is based on the corresponding proof in
\cite[Theorem 1 (3)]{ElekSzabo} for sofic groups $N$.

By assumption, the normal subgroup $N$ of $G$ is $\C$-approximable using 
a weight function $\delta$ that takes a constant value $c$ on all elements
of  $N \setminus \{1\}$. Since we can reduce the value of $c$ without affecting
the $\C$-approximability of $N$, we may assume that $c < 1$.
If $N \neq \{1\}$  then we define the weight function $\beta$ of $G$
by $\beta_g =c$ for all $g \neq 1$, and
if $N=\{1\}$, then we define $\beta$ by  $\beta_g=\frac{1}{2}$ for all $g \neq 1$. 

For  $g \in G$, let $\bar{g}$ be the homomorphic image of $g$ in $G/N$ and let
$\sigma:G/N \to G$ be a section (so $\overline{\sigma(h)} = h$
for all $h \in G/N$), where $\sigma(\bar{1}) = 1$.
We can lift $\sigma$ to a map from $G$ to $G$ for which the image of
$g \in G$ is $\sigma(\overline{g})$; we shall abuse notation and call that map
$\sigma$ as well.

To verify the $\C$-approximability condition on $G$,
let $F$ be a finite subset of $G$ and let $\epsilon > 0$.
We may assume that $\epsilon < \min(1/2,1-c)$.

The amenability of $G/N$ ensures the existence of a finite subset $\overline{A}$
of $G/N$ containing the identity element such that
$|\overline{A}\bar{g} \setminus \overline{A}| \le \epsilon|\overline{A}|$ for
all $g \in F \cup F^{-1} \cup F^2 \cup F^{-2}$. Let $A = \sigma(\overline{A})$;
note that all points of $A$ are fixed by the map $\sigma:G \rightarrow G$.
We define a map $\phi:G \to \Sym(A)$ as follows:
\[ {\rm for}\quad g \in G, a \in A, 
a^{\phi(g)} := \left \{ \begin{array}{ll} \sigma(ag), & \hbox{\rm if}\,\,
\overline{ag} \in \overline{A}\\
\hbox{\rm any choice with}\,\phi(g) \in \Sym(A),&\hbox{\rm otherwise.}\end{array}\right. \]

Let $E = N \cap (A \cdot F \cdot A^{-1})$.  
The $\C$-approximability of $N$ ensures the existence of an 
$(E,\epsilon,\delta,\ell_K)$-quasi-homomorphism $\psi: N \to K$ with $(K,\ell_K) \in \C$.

Now we let $W = K \wr \Sym(A) = \Sym(A) \ltimes B$ and define
$\Phi: G \to W$ by $\Phi(g) = (\phi(g),b)$ where, for $a \in A$,
$b(a) = \psi(\sigma(ag^{-1})ga^{-1})$.

We show first that $\newell_K^A(\Phi(g)) \ge \beta_g$ for $g \in F$.
If $g \not\in N$ then, since $\phi(g)$ moves all points $a \in A$ for which
$\overline{ag} \in \overline{A}$, we have
$\newell_K^A(\Phi(g)) \ge 1-\epsilon > 1/2 = \delta_g$.
If $g \in N \setminus \{1\}$ then $\overline{ag^{-1}}=\overline{a}$, so $\sigma(ag^{-1}) = a$ for all
$a \in A$, and $\newell_K^A(\Phi(g))$ is the average over $a \in A$
of $\ell_K(\psi(aga^{-1}))$. But since each $aga^{-1}  \in E \setminus \{1\}$,
these all exceed $\delta_g$.

Now let $g,h \in F$.
We aim to show that
           $\newell_K^A(\Phi(gh)\Phi(h)^{-1}\Phi(g)^{-1}) \leq 5\epsilon$.

For $a \in A$, we have
\[
\begin{array}{lllcl}
\Phi(g) &= (\phi(g),b),&\mbox{where}\quad b(a)&=&\psi(\sigma(ag^{-1})ga^{-1})\\
\Phi(h) &= (\phi(h),c),&\mbox{where}\quad c(a)&=&\psi(\sigma(ah^{-1})ha^{-1})\\
\Phi(gh) &= (\phi(gh),d),&\mbox{where}\quad d(a)&=&
                             \psi(\sigma(ah^{-1}g^{-1})gha^{-1})\\
\Phi(g)\Phi(h) &= (\phi(g)\phi(h),b^{\phi(h)}c),\\
&\mbox{where }
(b^{\phi(h)}c)(a) &= b^{\phi(h)}(a)c(a)&=& b(a^{\phi(h)^{-1}})c(a)\\
&& \multicolumn{3}{l}{=\psi(\sigma(a^{\phi(h)^{-1}}g^{-1})ga^{-\phi(h)^{-1}})
  \,\times \psi(\sigma(ah^{-1})ha^{-1}),}
\end{array}
\]
(where, for $a,k \in G$, we write $a^{-k}$ as shorthand for $(a^{-1})^k=(a^k)^{-1}$).
Then
\begin{eqnarray*}
\Phi(gh)(\Phi(g)\Phi(h))^{-1} &=& (\phi(gh),d)(\phi(g)\phi(h),b^\phi(h)c)^{-1}\\
&=& (\phi(gh),d)((\phi(g)\phi(h))^{-1},(b^\phi(h)c)^{-(\phi(g)\phi(h))^{-1}})\\
&=& (\phi(gh)(\phi(g)\phi(h))^{-1},(d(b^\phi(h)c)^{-1})^{(\phi(g)\phi(h))^{-1}}).
\end{eqnarray*}

Now, for a proportion of at least $1 - 2\epsilon$ of the points $a \in A$,
we have both $\overline{ah^{-1}}\in \overline{A}$ and
$\overline{ah^{-1}g^{-1}}\in \overline{A}$.
For those points $a$, we have
$a^{\phi(h)^{-1}} = \sigma(ah^{-1})$ 
and so 
the final expression for 
$(b^{\phi(h)}c)(a)$ above becomes 
\[ \psi(\sigma(ah^{-1}g^{-1})g \sigma(ah^{-1})^{-1})
         \times \psi(\sigma(ah^{-1})ha^{-1}), \]
and we see that the image of $a$ under the second component of
$\Phi(gh)(\Phi(g)\Phi(h))^{-1}$ is equal to a conjugate of 
\[ \psi(xy)\psi(y)^{-1}\psi(x)^{-1},\]
where $x=\sigma(ah^{-1}g^{-1})g \sigma(ah^{-1})^{-1}$
and $y = \sigma(ah^{-1})ha^{-1}$. The elements $x,y$ are both in the finite subset $E$ of $G$, and hence,
since $\psi$ is a quasi-homomorphism,
$\ell_K(\psi(xy)\psi(y)^{-1}\psi(x)^{-1}) < \epsilon$, and
we deduce that 
\[ \ell_K((d(b^{\phi(h)}c)^{-1})^{(\phi(g)\phi(h))^{-1}})(a)) < \epsilon, \]
for at least a proportion $1-2\epsilon$ of the points of $A$.

Our choice of $A$ ensures also that  
$\phi(gh)(\phi(g)\phi(h))^{-1}(a)=a$ for at
least a proportion $1-2\epsilon$ of the points $a$ of $A$.

Now, for at least a proportion $1-4\epsilon$ of the points of $A$,
the conditions of both of the last two paragraphs hold, and so
we can deduce
\[ \newell_K^A(\Phi(gh)\Phi(h)^{-1}\Phi(g)^{-1}) < \epsilon(1-4\epsilon) + 4\epsilon < 5 \epsilon.\]
\end{proof}

In particular, by taking $\C= \F$ with each $K \in \F$ associated with all possible length functions,
we see that that the class of weakly sofic groups is closed under extension by amenable groups.

In general, $\ell_K$ commutator-contractive does not imply that $\newell_K^A$
is commutator-contractive. 
But if, instead, we define $\ell_K^A$ as we did in
Section \ref{sec:wp} (that is,
for $b \in B$, $\ell_K^A(b) = \max_{a \in A} \ell_K(b(a))$, and
$\ell_K^A(\alpha b) = 1$ when $1 \ne \alpha \in \Sym(A)$) then, as we proved
in Lemma \ref{lem:wpcc}, $\ell_K^A$ is commutator-contractive.

Our proof of Theorem \ref{thm:amenable} does not always work with
this commutator-contractive norm, but it does work if $\phi:G/N \to A$ is a
homomorphism. In particular, when $G/N \cong (\Z,+)$, we can choose
$A$ to be $\{ x\in \Z : -m \le x \le m \}$ for some $m$ and define $\phi$
to be addition modulo $2m+1$. So, by applying this repeatedly, we have

\begin{proposition} The class of $\F_c$-approximable groups
is closed under extension by polycyclic groups.
\end{proposition}

\textsc{D. F. Holt,
Mathematics Institute,
University of Warwick,
Coventry CV4 7AL,
UK
}

\emph{E-mail address}{:\;\;}\texttt{D.F.Holt@warwick.ac.uk}

\textsc{Sarah Rees,
School of Mathematics and Statistics,
University of Newcastle,
Newcastle NE3 1ED,
UK
}

\emph{E-mail address}{:\;\;}\texttt{Sarah.Rees@newcastle.ac.uk}


\begin{thebibliography}{1}

\bibitem{ArzhantsevaGal} G.N. Arzhantseva and S. Gal, On approximation properties of semi-direct products of groups, 
\verb+http://arxiv.org/abs/1312.7682+.

\bibitem{ArzhantsevaPaunescu}
G.N. Arzhantseva and L. Paunescu,
Linear Sofic Groups and Algebras, Trans. Amer. Math. Soc. (2016) in press,\\
\verb+http://arxiv.org/abs/1212.6780+.

\bibitem{BrownDykemaJung} N. Brown, K. Dykema and K. Jung, Free entropy dimension in amalgamated free products, Proc. London Math. Soc., 97(2) (2008) 339--367.

\bibitem{Capraro} V. Capraro and M. Lupini, Introduction to sofic and
hyperlinear groups and Connes' embedding conjecture,
\verb+http://arxiv.org/abs/1309.2034+.

\bibitem{CiobanuHoltRees} L. Ciobanu, D.F. Holt and S. Rees, Sofic groups; graph products and graphs of groups, 
Pac. J. Math. 271 (2014) 53--64

\bibitem{DixonMortimer}
J.D. Dixon and B. Mortimer,
{\em Permutation Groups}.
Graduate Texts in Mathematics, 163. Springer, New York, 1996.

\bibitem{ElekSzabo3}
G. Elek and E. Szabo, Hyperlinearity, essentially free actions and
$L^2$-invariants. The sofic property, Math. Ann., 332 (2005), 421–-441.

\bibitem{ElekSzabo} G. Elek and E. Szabo, On sofic groups, J. Group Theory 9
(2006), 161 --171.

\bibitem{ElekSzabo2} G. Elek and E. Szabo, Sofic representations of amenable
groups,
Proc. Amer. Math. Soc. 139 (2011), 4285--4291.

\bibitem{Glebsky} L. Glebsky, Characterizations of sofic groups and equations
over groups,
\verb+http://arxiv.org/abs/1405.7329+.

\bibitem{HayesSale}
B. Hayes and A. Sale, The wreath product of two sofic groups is sofic,
\verb+http://arxiv.org/abs/1601.03286+.

\bibitem{Higman} G. Higman, A finitely generated infinite simple group, J. London Math. Soc. 26 (1951) 61--64.

\bibitem{Paunescu} E. Paunescu, On sofic actions and equivalence relations,
J. Funct. Anal. 261 (2011) 2461???-2485.

\bibitem{PestovKwiatkowska}
V. Pestov and A. Kwiatkowska,
An introduction to hyperlinear and sofic groups,
\verb+http://arxiv.org/abs/0911.4266v2+.

\bibitem{Popa} S. Popa, Free-independent sequences in type $II_1$ factors and
related problems, Recent advances in operator algebras (Orl\'eans, 1992),
Ast\'erisque 232 (1995) 187--202.

\bibitem{Radulescu}
F. R\u{a}dulescu, The von Neumann algebra of the non-residually finite
Baumslag-Solitar group $ab^2 a^{-1} = b^3$ embeds in to $\mathcal{R}^\omega$,
In Hot topics in operator theory, number 9, pages 173–-185.
Theta Ser. Adv. Math, Theta, Bucharest, 2008.

\bibitem{Stolz} A. Stolz, Properties of linearly sofic groups, \\
\verb+http://arxiv.org/abs/1309.7830v1+.

\bibitem{Thom} A. Thom, About the metric approximation of Higman's group,
J. Group Theory 15 (2012), 301--310.

\bibitem{Voiculescu}
D. Voiculescu, A strengthened asymptotic freeness result for random matrices
with applications to free entropy,
Internat. Math. Res. Notices, (1) (1998), 41--63.

\end{thebibliography}
\end{document}